\documentclass[11pt]{article}

\usepackage{fullpage}
\usepackage{appendix}
\usepackage{amssymb}
\usepackage{amsfonts}
\usepackage{amsmath}
\usepackage{amsthm}
\usepackage{latexsym}
\usepackage{epsfig}
\usepackage{makeidx}
\usepackage{graphicx}
\usepackage{hyperref}
\usepackage{epstopdf}
\usepackage{setspace}
\usepackage{xspace}
\usepackage{authblk}
\usepackage{color}
\usepackage[dvipsnames]{xcolor}

\definecolor{corlinks}{RGB}{0,0,150}
\definecolor{cormenu}{RGB}{0,0,150}
\definecolor{corurl}{RGB}{0,0,150}

\hypersetup{
colorlinks=true,
urlcolor=corlinks,
linkcolor=corlinks,
menucolor=cormenu,
citecolor=corlinks,
pdfborder= 0 0 0
}

\newtheorem*{notation}{Notation}
\newtheorem{theorem}{Theorem}[section]
\newtheorem{lemma}[theorem]{Lemma}

\newtheorem{definition}[theorem]{Definition}
\newtheorem{corollary}[theorem]{Corollary}
\newtheorem{proposition}[theorem]{Proposition}
\newtheorem{claim}[theorem]{Claim}

\newtheorem*{theorem*}{Theorem}
\newtheorem*{proposition*}{Proposition}

\let\:=\colon
\let\epsilon\varepsilon

\def\eps{\epsilon}
\def\hknp{\mathcal{H}^k(n,p)}

\DeclareMathOperator{\Bin}{Bin}

\newcommand{\ekr}{Erd\H{o}s-Ko-Rado\xspace}

\newcommand{\cF}{\mathcal{F}}

\newcommand{\cB}{\mathcal{B}}
\newcommand{\cH}{\mathcal{H}}

\newcommand{\pr}{\mathbb{P}}
\newcommand{\bG}{G}
\newcommand{\ex}{\mathbb{E}}

\begin{document}

\title{Erd\H{o}s-Ko-Rado for random hypergraphs: asymptotics and stability\\~\\}

\author[1]{Marcelo M. Gauy}
\author[2]{Hi\d{\^{e}}p H\`{a}n}
\author[3]{Igor C. Oliveira}
\affil[1]{\small Institute of Theoretical Computer Science, ETH Z\"{u}rich.\footnote{A significant fraction of this work was completed while the author was at University of S\~{a}o Paulo. Supported by CNPq grants 132238/2012-8 and 248952/2013-7.}}
\affil[2]{\small Instituto de Matem\'{a}ticas, Pontificia Universidad Cat\'olica de Valpara\'iso.\footnote{{The  author was supported by the Millennium Nucleus Information and Coordination in Networks ICM/FIC RC130003, 
the FONDECYT Iniciaci\'on grant  11150913, FAPESP (Proc.~2010/16526-3 and 2103/03447-6), BEPE (2013/11353-1) and PROBRAL CAPES/DAAD Proc. 430/15. 
}}}
\affil[3]{\small Faculty of Mathematics and Physics, Charles University in Prague.\footnote{Parts of this work were completed while the author was at Columbia University and as a visitor at University of S\~{a}o Paulo. Supported in part by CNPq grant 200252/2015-1.}}

\maketitle
\begin{abstract}
We investigate the asymptotic version of the Erd\H{o}s-Ko-Rado theorem for 
the random $k$-uniform hypergraph $\hknp$. For $2 \leq k(n) \leq n/2$,  
let $N=\binom{n}k$ and $D=\binom{n-k}k$. We show that with probability tending to 1 as 
$n\to\infty$, the largest intersecting subhypergraph of 
$\hknp$ has size $(1+o(1))p\frac kn N$, for any 
$p\gg \frac nk\ln^2\!\left(\frac nk\right)D^{-1}$. This lower bound on $p$ is asymptotically best possible for $k=\Theta(n)$. For this range of $k$ and $p$, we are able to show stability as well. 

A different behavior occurs when $k = o(n)$. In this case, the lower bound on $p$ is almost optimal. Further, for the small interval $D^{-1}\ll p \leq (n/k)^{1-\varepsilon}D^{-1}$, the largest intersecting subhypergraph of 
$\hknp$ has size $\Theta(\ln (pD)N D^{-1})$, provided that $k \gg \sqrt{n \ln n}$. 

Together with previous work of Balogh, Bohman and Mubayi, these results settle the asymptotic size of the largest intersecting family in $\hknp$, for essentially all values of $p$ and $k$.
\end{abstract}

\section{Introduction}\label{s:introduction}

The \ekr theorem \cite{EKR}  
is a cornerstone in extremal combinatorics.
Let $[n]$ denote the set $\{1,2,\dots,n\}$, and $\binom{[n]}k$ denote the set of all $k$-element subsets of $[n]$.
A  family of $k$-element sets  $\mathcal{F}\subset \binom{[n]}k$ is called a $k$-uniform hypergraph on the vertex set $[n]$, and such a hypergraph 
is called \emph{intersecting}  if $A \cap B \neq \emptyset$ holds for every edge $A,B \in \mathcal{F}$.
The \ekr theorem then states that for $2\leq k\leq n/2$, an intersecting family $\cF\subset\binom{[n]}k$ must satisfy $|\mathcal{F}| \leq \frac kn\binom{n}{k}$. 
This is best possible, as seen by the \emph{principal} hypergraphs $\cF_i$, 
which consist of all edges containing the fixed element $i\in[n]$.

We investigate a random analogue of the \ekr theorem in which the ambient space $\binom{[n]}k$ in the  theorem is replaced by 
 a random space. Random analogues of extremal results have been studied extensively in the last decades, and we refer to \cite{Schacht,ConlonGowers,BMS,SaxtonThomason} for the history of this line of research and recent breakthroughs.

The ambient random space we will work with is $\hknp$, 
the binomial random $k$-uniform hypergraph on the vertex set $[n]$ in 
which  each edge $e \in \binom{[n]}{k}$ is included in $\hknp$ independently with probability $p$. 
Further,  for a $k$-uniform $\cH$, let $i(\cH)$ denote the size of the largest intersecting subhypergraph of $\cH$, i.e., $i(\cH)=\max\{|\cF|\colon \cF\subset H\text{ and $\cF$ is intersecting}\}$. In this notation the \ekr theorem states that
$i(\cH^k(n,1))=i\left(\binom{[n]}k\right)=\frac kn\binom{n}k$.

\begin{notation}
All asymptotic limits in this paper are taken as $n \rightarrow \infty$. If we write $a(n) \ll b(n)$ or $a(n) = o(b(n))$, it means that $a(n)/b(n) \rightarrow 0$. In particular, the notation $o(1)$ represents a function that goes to $0$ as $n \rightarrow \infty$, as usual. For simplicity, we omit floor and ceiling functions, whenever they are not essential. We say that a sequence of events $\mathcal{E}_n$ holds asymptotically almost surely if $\Pr[\mathcal{E}_n] \to 1$ as $n \to \infty$. By $\ln^d{c}$ we denote  $(\ln c)^d$.
\end{notation}

We will be interested in $i(\hknp)$ for $k=k(n)$ and all $p=p(n)\in (0,1)$. This question was
investigated by Balogh, Bohman and Mubayi \cite{BBM}, which obtained  very precise results 
on the size and on the structure of the largest intersecting family in 
$\hknp$, for 
$k\leq n^{1/2-o(1)}$. 
For larger $k$, they obtained asymptotic tight bounds on $i(\hknp)$, 
however, only for rather large 
values of $p$. 
In general, their result highly depends on the range of $k$ and $p$, and hence it is slightly cumbersome to state. 
Therefore, we will only partially discuss it here, and refer to \cite{BBM} for detailed information. 
Their result concerning the large range of $k$ is given below in Proposition \ref{prop:BBM}.

\begin{proposition}[Proposition 1.3 in \cite{BBM}]
\label{prop:BBM} Let $\delta = \delta(n) > 0$ and 
$N=\binom nk$. If $\,\ln n \ll k < (1 - \delta)n/2$ and $p \gg (1/\delta) ((\ln n) /k)^{1/2}$, then almost surely as $n \rightarrow \infty$\emph{:}
$$
i(\hknp) = (1 + o(1))p (k/n)N.
$$
\end{proposition}

 In other words, for this range of $p$, the expected size
 of the intersection of a principal family $\cF_i$ with $\hknp$  is very close to the size of a 
 maximum intersecting subfamily of $\hknp$.
We extend this result, and provide an almost complete description of $i(\hknp)$ as follows.
\begin{theorem}\label{thm:ekr}
For all $0 < \eps < 1$ there exists a constant $C>0$ for which the following holds.
Let $p = p(n) \in (0,1)$, $k = k(n)$, where $2 \leq k \leq n/2$, $N = \binom{n}{k}$, and $D = \binom{n-k}{k}$. 
Then almost surely as $n \rightarrow \infty$\emph{:}
\begin{enumerate}
 \item[\emph{(1)}]  $i(\hknp)=(1\pm\eps)pN$ \quad \quad \quad $\,\,$ if  $N^{-1} \ll  p  \ll D^{-1}$,
 \item[\emph{(2)}]  $i(\hknp)\geq (1-\epsilon)\frac{N} D\ln(pD)$ \quad  if  $D^{-1} \ll  p\leq (n/k)D^{-1}$ and $\,k \gg \sqrt{n \ln n}$,
 \item[\emph{(3)}]  $i(\hknp)\leq C \frac{N}D\ln(pD)$ \quad \quad \quad if  $D^{-1} \ll  p  \leq (n/k)^{1-\epsilon} D^{-1}$,
 \item[\emph{(4)}]  $i(\hknp)=(1 \pm \eps)p \frac knN$ \quad \quad $\,\,$ if $p \geq C (n/k) \ln^2 (n/k) D^{-1}.$
\end{enumerate}
\end{theorem}

The first bound follows from a standard deletion argument, and we state it here for completeness.
Note also that $i(\hknp)$ is monotone in $p$, hence, in the range of $p$ around $(n/k)D^{-1}$ not
mentioned in the theorem, we have $i(\hknp)=O((N/D)\ln^2(n/k))$ due to (4).

If $k$ is linear in $n$, the bounds in (1) and (4) 
determine $i(\hknp)$ asymptotically for essentially all $p$. Here, we have a change of behaviour 
around $D^{-1}$. Roughly speaking, for $p$ below $D^{-1}$, essentially  all 
of $\hknp$ is intersecting. Beyond that point, i.e.\ for $p\gg D^{-1}$, the largest intersecting
subhypergraph of $\hknp$ has size very close to the size of the
intersection of a principal hypergraph with $\hknp$. Observe that cases (2) and (3) are trivial for $k = \Theta(n)$.

For $k=o(n)$ there is a rather short range of $p$ where 
$i(\hknp)$ reveals a ``flat'' behaviour. Indeed,
the upper bound (3) shows that $i(\hknp)$ grows slowly with $p$, since
it appears only in the $\ln$-term. 
The corresponding lower bound in (2) shows that
for $k\geq n^{1/2+o(1)}$ this bound is tight up to a multiplicative constant. We provide
no  lower bound for the range $k<n^{1/2-o(1)}$ here, as in this case 
the result of Balogh et al. is more satisfactory. Again, we refer to \cite{BBM} for further information.

Although the ``flat range'' phenomenon might come as a surprise, it has been observed elsewhere.
Indeed, in the dense case, i.e. for $p=1$, and for $k=o(n)$,
the size of the largest intersecting family is vanishing compared to the ambient space, that is, 
$i(\binom{[n]}k)=\frac kn\binom{n}k =o\left (\binom{n} k \right )$. For these so called ``degenerate'' problems,
the random analogues typically reveal such an intermediate flat behaviour, as observed for example in \cite{KKS,KLRS}.

The question of for which range of $p$ the largest intersecting family $\cF\subset\hknp$ is indeed the 
projection of a principal family has been successfully addressed in \cite{BBM} for 
$k<n^{1/2-o(1)}$. For larger $k$, which we are mainly interested in, the problem seems to be more 
complicated, and has only been 
studied recently in \cite{HammKahn1}, for constant $p$. 
We make no contribution to this question here. However,
besides the bounds on $i(\hknp)$, we are able to show stability for $k=\Theta(n)$ in the same 
range for $p$ as in case (4) in Theorem~\ref{thm:ekr}.

\begin{theorem}
\label{thm:stability}
For every $\beta>0$ and $\eps>0$ there exist constants $\delta>0$ and $C>0$ for which the following holds. For any 
$\beta n< k(n) <(1/2-\beta)n$ and $p \geq C \cdot D^{-1}$, asymptotically almost surely
stability holds,
i.e., for every intersecting family $\cF \subset \hknp$ of size
 $|\cF| \geq (1-\delta)p(k/n)N$, there is an element $i\in[n]$ that is contained in all but at most $\eps p (k/n)N$ elements of $\cF$. 
\end{theorem}

In the dense case, i.e.\ for $p=1$, the result was proven by Friedgut \cite{Friedgut}. 
Indeed, the proof of Theorem~\ref{thm:stability} 
relies on the result of Friedgut and on a removal lemma for the Kneser graph due to 
Friedgut and Regev~\cite{FriedgutRegev}.

\paragraph{Further results and organization.} 

In proving Theorems ~\ref{thm:ekr} and \ref{thm:stability}, 
it will be convenient for us to work with 
the Kneser graph $K(n,k)$. The vertex set of this graph is $\binom{[n]}k$, and two
$k$-element sets form an edge if and only if they are disjoint. Hence 
$K(n,k)$ is a $\binom{n-k}k$-regular graph on $\binom{n}k$ vertices, and 
a hypergraph $\cF\subset \binom{[n]}k$ is intersecting if and only if 
$\cF$ is an independent set in $K(n,k)$.
Further, let $K(n,k,p)$ denote the subgraph
of $K(n,k)$ induced on the random \emph{vertex} set obtained by including each vertex 
from $\binom{[n]}k$ independently with probability $p$. Due to the correspondence, all bounds
on intersecting subgraphs of $\hknp$ will follow from corresponding bounds on the size of  
largest independent sets in $K(n,k,p)$.

Using this translation, Theorem~\ref{thm:ekr}  
follows from a more general scheme which relies on
the technical Proposition~\ref{prop:transference} and Lemma \ref{lem:hoffman},
to be introduced in the next section. Further, for Theorem~\ref{thm:stability}
we will need Lemma \ref{lem:robuststability}, which together with Lemma \ref{lem:hoffman} 
will be proven in Section~\ref{sec:kneser}.
Based on these results, we will give the proofs of Theorems \ref{thm:ekr}
and \ref{thm:stability} in Section \ref{sec:proofmain}.

In general, the proof scheme based on Proposition~\ref{prop:transference}
and Lemma \ref{lem:hoffman} can be used to bound the size of the largest independent sets 
in random subgraphs of any $D$-regular graph $G$ (actually, a sequence of graphs). 
Here by random subgraph we mean the graph induced on a binomial random subset of the vertex set. 
This application yields asymptotically sharp bounds if $G$
has an independent set of size (close to) $-\lambda_{\min} |V(G)|/(D-\lambda_{\min})$.
Indeed, Theorem \ref{thm:ekr} shows such 
an application to the Kneser graph, and there are many other graphs for which this applies.
We refer, e.g., to \cite{ADFS} for a list of such graphs which include the
weak product of the complete graph,
line graphs of regular graphs which contain a perfect matching, Paley graphs, 
some strongly regular graphs, and
appropriate classes of random regular graphs (see Section 5.1. of \cite{ADFS}).

The proof of Proposition~\ref{prop:transference} will be given in Section~\ref{sec:transference}.
It is based on a description of all independent
sets  in locally dense graphs. 
This idea can be traced back to the work of Kleitman and Winston \cite{KleitmanWinston}, and has 
been exploited in various contexts since their work.
Though similar proofs have been given elsewhere, none of them seems to fully fit in our context. 
This also applies to the powerful extension of the ideas of Kleitman and Winston to hypergraphs 
due to Balogh, Morris and Samotij in~\cite{BMS} (see also \cite{SaxtonThomason}), which only partially 
suits our needs.

\section{Proofs of Theorems \ref{thm:ekr} and \ref{thm:stability}}
\label{sec:proofmain}

As mentioned before, the proofs of the main theorems rely on Proposition~\ref{prop:transference}.
A central notion employed in this proposition which applies to $K(n,k)$ 
is the following.
\begin{definition}\label{def:supersat}
Given $\lambda \in (0,1]$, $\gamma \in  (0,1]$, and a graph $G$ on $N$ vertices, 
 we say that $G$ is  $(\lambda$,$\gamma)$-\emph{supersaturated} 
if for any subset $S \subseteq V(G)$ with $|S| \geq \lambda N$, we have
\[e(S)\geq \gamma \left (\frac{|S|}{N} \right )^2\! e(G).\]
In addition, let $\lambda=\lambda(n)>0$ and $\gamma=\gamma(n)>0$. A sequence 
$\{G_n\}_{n \in \mathbb{N}}$ is called $(\lambda(n), \gamma(n))$-\emph{supersaturated} 
if $G_n$ is 
$(\lambda(n), \gamma(n))$-\emph{supersaturated} for each $n \in \mathbb{N}$.
\end{definition}
Hence, in a $(\lambda,\gamma)$-supersaturated graph $G$ each set $S$ of size at least $\lambda N$ spans many edges. 
Indeed, up to the multiplicative factor $\gamma$, 
$S$ spans as many edges as expected from a random subset of $V(G)$ of the 
same size.

Using an extension of Hoffman's spectral bound \cite{Hoffman}, 
one can relate supersaturation to the eigenvalues of a graph. We refer to Section~\ref{sec:kneser}
for the proof.
\begin{lemma} \label{lem:hoffman}
Let $G$ be a $D$-regular graph on $N$ vertices, and let 
$\lambda_{\mathrm{min}}$ denote the smallest eigenvalue of the adjacency matrix of $G$.
Then every set $S \subset V(G)$ satisfies
\[
e(S) \geq 
\left( \frac{\lambda_{\min}}{D}\frac N{|S|}+\frac{D-\lambda_{\min}}D\right) \left(\frac{|S|}{N}\right)^2\! e(G).
\]
\end{lemma}

As the eigenvalues of the Kneser graph are known due to Lov\'asz~\cite{Lovasz}, we immediately 
conclude the following supersaturation for the Kneser graph.

\begin{lemma}\label{lem:supersaturation}
Let $2 \leq k \leq n/2$ and $\tau=\tau(n)>0$. Then $K(n,k)$ is
$\left((1 + \tau)\frac{k}{n}, \frac{\tau}{1 + \tau}\right)$-supersaturated.
\end{lemma}

\begin{proof}
The Kneser graph $K(n,k)$ has degree $D=\binom{n-k}k$, and
the smallest eigenvalue of $K(n,k)$ is 
given by (see \cite{Lovasz}):
\[\lambda_{\mathrm{min}}=- \binom{n-k-1}{k-1} = - \frac{k}{n-k} D.\] 
Let $S\subset \binom{[n]}k$ be of size at least $(1+\tau)\frac knN$, with $N=\binom{n}k$. Lemma~\ref{lem:hoffman} implies that  
$$e(S)\geq \left( -\frac{n}{(n-k)(1+\tau)}+\frac{n}{n-k}\right) \left(\frac{|S|}{N}\right)^2\! e(G),$$
and the claim follows.
\end{proof}

Beyond the notion of supersaturation needed for the proof of Theorem~\ref{thm:ekr}, we will rely on the following notion of 
robust stability in the proof of Theorem~\ref{thm:stability} (see also \cite{DBLP:journals/rsa/Samotij14}).

\begin{definition} \label{def:stability}
Let $\lambda, \varepsilon, \delta>0$. Let $G$ be a graph on $N$ vertices, 
and let $\mathcal{B}(G)\subseteq \mathcal{P}(V(G))$ be a family of sets.
We say that $G$ is $(\lambda,\mathcal{B}(G))$-stable with respect to 
$(\varepsilon,\delta)$ if for every
$S\subseteq V(G)$ with $\left|S\right|\geq (1-\delta)\lambda N$, we have  either 
\begin{itemize}
\item $e(S)\geq \delta  \left(\frac{\left|S\right|}{N}\right)^2\! \cdot e(G)$, or
\item $| S \, \backslash \, B |\leq \varepsilon\lambda N$, for some $B\in \mathcal{B}(G)$.
\end{itemize}
\noindent In addition, let $\lambda=\lambda(n)>0$, $\left\{G_n\right\}_{n\in\mathbb{N}}$ be a sequence of graphs, and
$\mathcal{B}=\{\mathcal{B}_n\}_{n \in \mathbb{N}}$ with $\mathcal{B}_n\subset\mathcal{P}(V(G_n))$
be a sequence of families of sets. We say that 
$\left\{G_n\right\}_{n\in\mathbb{N}}$ is 
$(\lambda,\mathcal{B})$-stable if for any $\varepsilon>0$
there exists $\delta>0$ and $n_0 \in \mathbb{N}$ such that for all $n \geq n_0$, the graph 
$G_n$ is $(\lambda(n),\mathcal{B}_n)$-stable with respect to $(\varepsilon,\delta)$.
\end{definition}
It is instructive to think of $\cB(G)$ as the family of largest independent sets in $G$, and of $\lambda N$ as
the size of each $B\in\cB$.  The first part of the definition roughly says  
that if $G$ is robustly stable, then any vertex set 
$S$ whose size is close to the size of a largest independent set in $G$ must either contain many edges,
or be close to a largest independent set in structure.

The Kneser graph satisfies robust stability for  $k$ linear in $n$, as stated in the next lemma. 
It is a direct consequence
of the corresponding stability result proven by Friedgut~\cite{Friedgut}, and the removal lemma
proven by Friedgut and Regev~\cite{FriedgutRegev}. Again, we refer to Section~\ref{sec:kneser} for 
the details of the proof. In the following, let $\cF_i\subset\binom{[n]}k$ denote the principal hypergaph centered at $i$, i.e., the hypergraph  
consisting of all $k$-element subsets of $[n]$ containing $i\in [n]$.

\begin{lemma}\label{lem:robuststability}
Let $\beta>0$ and $k=k(n)$, where $\beta n \leq k \leq (1/2 - \beta) n$, and let $G_n = K(n,k)$.  Further, let  
$\mathcal{B}_n(G_n) = \{ \mathcal{F}_i \mid i \in [n]\} \subset \mathcal{P}(V(G_n))$, and set 
$\mathcal{B} = \{\mathcal{B}_n\}_{n \in \mathbb{N}}$. Then 
$G = \{G_n\}_{n \in \mathbb{N}}$ is $(k/n,\mathcal{B})$-stable.
\end{lemma}

With supersaturation and robust stability defined, we are now ready to state our main technical 
result. Given a graph $H$, we use $\alpha(H)$ to denote the size of the largest independent set in $H$. Also, for a finite set $V$, we let $V_p$ be a random subset of $V$ obtained by selecting each element $v \in V$ independently with probability $p$.

\begin{proposition}\label{prop:transference} 
Let $\lambda = \lambda(n)$ and $\gamma = \gamma(n)$ be $(0,1)$-valued functions, and let $G = \{G_n\}_{n \in \mathbb{N}}$ be a family of graphs, where each $G_n$ has $N=N(n)$ vertices \emph{(}with $\lim_{n\to\infty}N(n)=\infty$\emph{)} and average degree $D=D(n)$. For any constant $0 < \varepsilon < 1$ there exist constants $C = C(\varepsilon) > 0$ and $\delta = \delta(\varepsilon) > 0$ such that for any probability sequence $p=p(n) \in (0,1]$, the following holds.  For a random spanning subgraph $H_n=G_n[V_p]$, where $V=V(G_n)$, we have:

\begin{enumerate}
\item[\emph{(}i\emph{)}] If $N^{-1}\ll p\ll D^{-1}$, then $\alpha(H_n)=(1\pm\eps)pN$ asymptotically almost surely.
\item[\emph{(}ii\emph{)}] If $\bG$ is $(\lambda, \gamma)$-supersaturated and 
$9D^{-1}\leq p \leq \lambda^{\varepsilon}(\lambda \gamma D)^{-1}$,
then 
\[\pr\left(\alpha(H_n)>\frac{4N}{\eps\gamma D}\ln(pD)\right)\leq 
\exp\left\{-\frac{N}{\gamma D}\ln (pD)\right\}. \]
\item[\emph{(}iii\emph{)}] If $\bG$ is $(\lambda, \gamma)$-supersaturated and 
$p \geq C (\lambda\gamma D)^{-1}\ln^2(e/\lambda)$, then
\[\pr\left(\alpha(H_n)\geq (1+\varepsilon)\lambda pN\right)\leq \exp\{-\eps^2p\lambda N/24\}.\]
\item[\emph{(}iv\emph{)}] If $\bG$ is $(\lambda, \mathcal{B})$-stable and 
$p \geq C (\lambda D)^{-1}\ln^2(e/\lambda)$, then 
with probability at least $1- \exp(-\delta^2\lambda p N/2)$, the following holds\emph{:}\ every independent set $I$ in $H_n$ of size at least $(1-\delta)\lambda pN$ satisfies $|I \, \backslash \, B| \leq \varepsilon \lambda pN$ for some $B\in\mathcal{B}_n$.
\end{enumerate}
\end{proposition}

In addition, the following result  will be needed for the lower bound (2) in
Theorem~\ref{thm:ekr}. It is Shearer's extension \cite{Shearer} of a result due to Ajtai, Koml\'os and Szemer\'edi~\cite{AKS}.

\begin{proposition}[\cite{AKS}, \cite{Shearer}]
\label{prop:aks}
Let $G = \{G_n\}_{n \in \mathbb{N}}$ be a sequence of graphs on $N = N(n)$ vertices with average 
degree at most $D = D(n)>1$. If each $G_n$ is 
triangle-free, then $G_n$ contains an independent set of size 
$N(D\ln D-D+1)/(D-1)^2 \geq N (-1 + \ln D)/D$. \qed
\end{proposition}

Finally, we shall repeatedly use Chernoff's bound for binomial random variables, which we state here
for reference (see~\cite[Theorem~2.1]{JLR_randomgraphs}).

\begin{lemma}\label{l:chernoff}
Given integers $m, s>0$ and $\zeta \in [0,1]$, we have\emph{:} 
\begin{align}
\label{eq:chernoffut}
\pr(\Bin(m,\zeta)\geq m\zeta+ s)&\leq e^{-s^2/(2\zeta m+s/3)}.\\
\label{eq:chernofflt}
\pr(\Bin(m,\zeta)\leq m\zeta- s)&\leq e^{-s^2/(2\zeta m)}.
\end{align}
\end{lemma}

We are now ready to present the proofs of Theorems \ref{thm:ekr} and \ref{thm:stability}.

\begin{proof}[Proof of Theorem~\emph{\ref{thm:ekr}}]
Given $0 < \eps < 1$, apply Proposition~\ref{prop:transference} with $\eps/4$ in order to obtain
a corresponding constant $C_1$. Let $C=\max\{32/\eps^2,32 C_1/\eps\}$. Further, let $k=k(n)$, and $G_n = K(n,k)$. Recall that $G_n$ is a $D$-regular graph on 
$N$ vertices, with $D=D(n)=\binom{n-k}k$
and $N=N(n)=\binom{n}k$. Let $H_n = G_n[V_p]$, where $V = V(G_n)$, and $V_p$ is the set
obtained by including each vertex of $V$ independently with probability $p$. 
We apply Proposition~\ref{prop:transference} to $\{G_n\}_{n\in\mathbb{N}}$, with functions 
$N(n)$ and $D(n)$ as defined above.
The first bound of Theorem~\ref{thm:ekr} follows immediately from the first case of Proposition~\ref{prop:transference}.

For the third and fourth bounds of Theorem~\ref{thm:ekr}, note that by Lemma~\ref{lem:supersaturation} applied with $\tau=\eps/2$, we know that $G_n$ is 
$(\lambda,\gamma)$-supersaturated, with $\lambda\leq (1+\eps/2) k/n$ 
and $\gamma=\eps/4$. Thus we can apply Proposition~\ref{prop:transference} in both cases. We start with the third bound of Theorem~\ref{thm:ekr}. Assume that 
$k=o(n)$, since for $k$ linear in $n$ this range of $p$ is trivial.
By the second part of Proposition~\ref{prop:transference} applied with $\eps_1=\eps/2$, we derive that
for $9D^{-1}\leq p \leq \left(n/k\right)^{1-\eps/2}(\eps D)^{-1}$, which contains our interval for $p$ in the third case, we have
\[i(\hknp)<\frac{8}{\eps_1^2}\frac ND\ln (pD)\leq C\frac{N}D\ln(pD)\]
with probability at least $(1-\exp(-4 N\ln(pD)/(\eps D)))$. As $p\gg D^{-1}$, this probability tends to one as $n$ goes to infinity, which gives the upper bound in the third case.

Next we show the fourth bound of Theorem~\ref{thm:ekr}. The lower bound follows by considering
the subhypergraph of $\hknp$ consisting of all hyperedges containing, say, the element $n$. Using the Chernoff
bound (Lemma \ref{l:chernoff}), we have with high probability that this (intersecting) subhypergaph has size at least  
$(1-\eps)p(k/n)N$.
For the upper bound, we apply the third bound of Proposition~\ref{prop:transference} 
with $\eps/4$ and $\lambda$, $\gamma$ as chosen above. 
Then, by the choice of $C$, 
we have $i(\hknp)\leq (1+\eps)\frac kn p N$
for $p\geq C(n/k)D^{-1}\ln^{2}(n/k)\geq C_1(\lambda \gamma D)^{-1} {\ln^2 (e/\lambda)}$, and the claim follows.

Finally, we prove the second bound of Theorem~\ref{thm:ekr}. Observe that this range of $p$ is nontrivial only if $k \ll n$. By Chernoff's 
bound, almost surely $H_n$ has at least $(1-\eps/32)pN$ vertices.
Further, $\mathbb{E}[e(H_n)]=p^2 ND/2$, and it is not hard to see that $\mathsf{Var}[e(H_n)]\leq 2p^3N^2D +p^2ND$.
By Chebyshev's inequality, we derive
$$\mathbb{P}\big(|e(H_n)-\ex(e(H_n))|\geq \varepsilon p^2 ND/32\big)\leq \frac{32^2\mathsf{Var}(e(H_n))}{\varepsilon^2p^4 N^2D^2}$$
which goes to zero by the choice of $p$. 
\begin{claim}
For $(n \ln n)^{1/2} \ll k \ll n$ and $p\leq (n/k)D^{-1}$, asymptotically almost surely the number of triangles in $H_n$ is at most  $\eps pN/32$.
\end{claim}
\begin{proof}
The expected number of triangles in $H_n$ is at most 
$p^3\binom{n}k\binom{n-k}k\binom{n-2k}k$.
Using Markov's inequality and $p\leq (n/k)D^{-1}$, the claim follows if we can show that  
$(n/k)^2  \binom{n-2k}k \ll \binom{n-k}k.$ 
Indeed, \[\binom{n-k}k \binom{n-2k}k^{-1} =\frac{(n-k)\dots(n-2k+1)}{(n-2k)\dots (n-3k+1)} \geq \left(\frac{n-k}{n-2k}\right)^k\geq \left(1+\frac kn\right)^k,\]
and using $(1+x)\geq \exp\{x-x^2\}$ for $0 < x < 1$, we obtain together with our assumption $(n \ln n)^{1/2} \ll k \ll n$ that
\[\binom{n-k}k \binom{n-2k}k^{-1} \geq \exp\{k^2/n-k^3/n^2\}\gg n^2 \geq (n/k)^2,\]
which completes the proof of the claim.
\end{proof}

Hence, by removing at most $\eps pN/32$ vertices, we obtain a triangle free graph with 
at least $(1-\eps/16)pN$ vertices, and no more than $(1/2+\eps/32)p^2 ND$ edges.
Consequently, this graph has average degree at most $(1+\eps/4) pD$, and
due to Proposition~\ref{prop:aks}, it contains an independent set of size 
\[\frac{(1-\eps/16)pN}{(1+\eps/4)pD}\big(\ln ((1+\eps/4)pD) -1\big) \geq (1-\eps)\frac ND\ln pD.\] This completes the proof.
\end{proof}

\begin{proof}[Proof of Theorem~\emph{\ref{thm:stability}}]
Let $\beta>0$ be fixed, and $\beta n\leq k\leq (1/2-\beta)n$. 
Again, let $G_n$ denote the Kneser graph $K(n,k)$.
Set $\lambda=k/n$, and for a given $n$, let $\mathcal{B}_n$ be the set of all 
principal hypergraphs $\mathcal{F}_i$, for $i=1,\ldots,n$. 
By Lemma \ref{lem:robuststability}, the family $G = \{G_n\}$ is $(\lambda,\mathcal{B})$-stable, where $\mathcal{B}=\{\mathcal{B}_n\}_{n\in\mathbb{N}}$. For a given $\varepsilon>0$, we apply Proposition~\ref{prop:transference} in order to obtain constants $C'$ and {$\delta >0$}. Since $k = \Theta(n)$, it is possible to choose {an appropriate constant $C$ such that $\delta$ and $C$ satisfy} the conclusion of the theorem, which completes the proof.
\end{proof}

\section{Proofs of Lemmas~\ref{lem:hoffman} and \ref{lem:robuststability}}
\label{sec:kneser}

As mentioned before, the proof of Lemma~\ref{lem:hoffman} is a straightforward extension
of Hoffman's bound~\cite{Hoffman}.

\begin{proof}[Proof of Lemma~\emph{\ref{lem:hoffman}}]
Given  a $D$-regular $G$ with $N$ vertices and smallest eigenvalue $\lambda_{\min}$, we need to show that for every non-empty 
$S\subset V(G)$, 
\[e_S=e(S)\geq \left(\frac{\lambda_{\min}}D\frac N{|S|}+\frac{D-\lambda_{\min}}D\right)
\left(\frac{|S|}{N}\right)^{2}\!e(G).\]

Let $M$ denote the adjacency matrix of $G$. 
For $x,y \in \mathbb{R}^N$, let $\langle x,y \rangle = \sum_{i=1}^N x_i y_i$. Also,
let $v_S$ be the $0/1$-characteristic vector of $S$. First note that 
$\langle v_S, M v_S \rangle = 2 e_S$. Since $M$ is a symmetric real matrix, it is diagonalizable by an orthonormal basis. Let $u_1, \ldots ,u_N$ be normalized eigenvectors of $M$ with corresponding eigenvalues $\lambda_1 \geq \lambda_2 \geq \ldots \geq \lambda_N = \lambda_\mathrm{min}$, respectively. Since $G$ is a $D$-regular graph, we have $u_1 = (1/\sqrt{N}, \ldots ,1/\sqrt{N})$ and $\lambda_1 = D$. Let $v_S = \sum_{i=1}^N a_i u_i$ be the expansion of $v_S$ by eigenvectors. We have
\[
2 e_S = \langle v_S, M v_S \rangle = \sum_{i=1}^N \lambda_i a_i^2 \geq \lambda_1 a_1^2 + \lambda_\mathsf{min} \sum_{i=2}^N a_i^2.
\]  
Now observe that $a_1 = \langle v_S, u_1 \rangle = |S| /\sqrt{N}$. In addition, $|S| = \langle v_S, v_S \rangle = \sum_{i=1}^N a_i^2$. Therefore,
\begin{eqnarray}\nonumber
2e_S & \geq & D \frac{|S|^2}{N} + \lambda_{\mathsf{min}} \left ( |S| - \frac{|S|^2}{N} \right)\\
&=& |S| \left(\lambda_{\mathrm{min}} + \frac{|S|}{N} \left( D - \lambda_{\mathrm{min}} \right) \right)\nonumber \\
& = & \left(\frac{|S|}{N}\right)^{\!2} {2e(G)} \left({\frac{\lambda_{\min}}D}\frac N{|S|}+ \left(1-\frac{\lambda_{\min}}D\right)\right), \nonumber
\end{eqnarray}
and the lemma follows.
\end{proof}

We now proceed to show robust stability for the Kneser graph for $k = \Omega(n)$. 
The proof is a direct consequence of stability due to Friedgut~\cite{Friedgut} and a removal lemma
for the Kneser graph due to Friedgut and Regev~\cite{FriedgutRegev}, which we state next.

\begin{proposition}[Friedgut \cite{Friedgut}]\label{p:ekrstability}
Given $\beta>0$, let $k=k(n)$ be a sequence of integers satisfying $\beta n\leq k\leq (1/2-\beta)n$. For all $\varepsilon>0$ there exists $\delta>0$ and $n_0$ such that, for all $n \geq n_0$, the following holds.
If $\mathcal{F}\subseteq \binom{[n]}{k}$ is an intersecting family of size at least $(1-\delta)\binom{n-1}{k-1}$, then there is $i\in [n]$ such that $\left|\mathcal{F}\setminus\mathcal{F}_i\right|\leq \varepsilon\binom{n-1}{k-1}$.
\end{proposition}

\begin{proposition}[Friedgut and Regev \cite{FriedgutRegev}]\label{p:ekrremovallemma}
Given $\beta>0$, let $k=k(n)$ be a sequence of integers satisfying $\beta n\leq k\leq (1/2-\beta)n$. Moreover, let $N=\binom{n}{k}$ and $D=\binom{n-k}{k}$.
For all $\varepsilon>0$ there exists $\delta>0$ and $n_0$ such that, for all $n \geq n_0$, the following holds.
Every family $\mathcal{F}\subseteq \binom{[n]}{k}$ which spans at most $\delta |\mathcal{F}|^2 (D/N)$  non-intersecting pairs can be made intersecting by removing at most $\varepsilon\binom{n-1}{k-1}$ elements from $\mathcal{F}$.
\end{proposition}

\begin{proof}[Proof of Lemma~\emph{\ref{lem:robuststability}}]
Given any $\varepsilon>0$, first let $\varepsilon_2 = \varepsilon/2$, and apply {Proposition \ref{p:ekrstability}} to get a corresponding $\delta_2 = \delta_2(\varepsilon_2)>0$. Now set $\varepsilon_1 = \min(\varepsilon/2, \delta_2/2)$, and use this time {Proposition \ref{p:ekrremovallemma}} in order to obtain an appropriate $\delta_1 = \delta_1(\varepsilon_1) > 0$. Finally, set $\delta = \min(\delta_1, \delta_2/2) = \delta(\varepsilon) > 0$. 

It follows that for any family $\mathcal{F}$ with $|\mathcal{F}| \geq (1 - \delta) \binom{n-1}{k-1}$ and {$e(\mathcal{F}) \leq \delta (|\mathcal{F}|/N)^2 (ND/2) \leq \delta_1 |\mathcal{F}|^2 (D/N)$} there exists an intersecting family $\mathcal{F}' \subseteq \mathcal{F}$ obtained from $\mathcal{F}$ by removing at most $\varepsilon_1 \binom{n-1}{k-1}$ of its elements such that
$$
|\mathcal{F}'| \, \geq  \, (1 - \delta - \varepsilon_1) \binom{n-1}{k-1} \, \geq \, (1 - \delta_2) \binom{n-1}{k-1}. 
$$
In addition, Proposition \ref{p:ekrstability} implies that for some $i \in [n]$, we have $|\mathcal{F}' \setminus \mathcal{F}_i| \leq \varepsilon_2 \binom{n-1}{k-1}$. Therefore, 
$$
|\mathcal{F} \setminus \mathcal{F}_i| \, \leq \, |\mathcal{F} \setminus \mathcal{F}'| + |\mathcal{F}' \setminus \mathcal{F}_i| \, \leq \, \varepsilon_1 \binom{n-1}{k-1} + \varepsilon_2 \binom{n-1}{k-1}  \,
\leq \, \varepsilon \binom{n-1}{k-1}, \nonumber
$$
which completes the proof.
\end{proof}

\section{Proof of Proposition~\ref{prop:transference}}\label{sec:transference}

We begin with the proof of a simple structural result for independent sets in graphs (Lemma \ref{lem:greedyalg}). For a given graph $G$, let $\mathcal{I}_G(t)$ denote the set of 
independent sets of $G$ of size exactly $t$, and $\mathcal{I}_G$ denote the set of all 
independent sets in $G$.

\begin{lemma}
\label{lem:greedyalg}
Let $G$ be a graph on $N$ vertices, and $\gamma>0$ be an arbitrary real number. In addition, let $0< \ell < t$ be integers.
Then, for every independent set $I \subset V(G)$ of size at least $t$, there is a sequence of vertices  
$x_1,\dots,x_\ell\in I$ and a  sequence of subsets 
$V(G)\supseteq X_1\supseteq \dots\supseteq X_\ell$ depending only on $x_1,\dots,x_\ell$ such that:
\begin{itemize}
\item $x_1,\dots,x_i\not\in X_{i}$ for all $i\leq \ell$,
\item $I\setminus\{x_1,\dots,x_{i}\}\subset X_i$ for all $i\leq \ell$.
\end{itemize}
Moreover, we have either
\begin{itemize}
\item[\emph{(}i\emph{)}] $|X_i|\leq\left(1-2\gamma\frac{e(G)}{N^2}\right) |X_{i-1}|  \text{ for all } 1 \leq i\leq \ell$, or
\item[\emph{(}ii\emph{)}] $e(G[X_i])<\gamma\frac{|X_i|^2}{N^2}e(G) \text{ for some } 1 \leq i\leq\ell$.
\end{itemize}    
\end{lemma}

\begin{proof}
Fix an independent set $I$ of size at least $t$. 
We need to define the required sequences $x_1,\dots,x_{\ell}$ and 
$X_1,\dots,X_{\ell}$. Assume that we have already chosen elements 
$x_1,\ldots, x_{i-1}\in~I$ and sets 
$V(G) = X_0\supset X_1\supset\ldots\supset X_{i-1}$ satisfying the conditions of our result. Observe that initially no element has been selected, and for convenience we set $X_0 = V(G)$.

Consider an ordering $(v_1,\dots,v_{|X_{i-1}|})$ of the vertices in $X_{i-1}$
which satisfies 
\[ |N(v_i)\cap\{v_{i+1},\ldots,v_{|X_{i-1}|}\}|\geq |N(v_j)\cap\{v_{i+1},\ldots,v_{|X_{i-1}|}\}|\]
for all $i<|X_{i-1}|$ and all $j>i$. Such an ordering clearly exists, since one can repeatedly choose (and remove) the 
vertex with highest degree in the remaining graph. In this case we say that this is a \emph{max-ordering} of the elements in $X_{i-1}$.

Let $j$ be the smallest index such that the vertex $v_j$ in the max-ordering of $X_{i-1}$
is contained in $I$. Such index must exist, since $I\setminus\{x_1,\dots x_{i-1}\}\subseteq X_{i-1}$ and $i-1<\ell<t\leq |I|$. 
We define $x_i=v_j$, and set $S=X_{i-1}\setminus \{v_1,\dots,v_j\}$.

If $\deg(v_j,S)<2\gamma |S|e(G)/N^2$ then we let $X_i=S$. Note 
that, due to the max-ordering and the definition of $v_j$, {every vertex $v\in X_i=\{v_{j+1},\ldots,v_{|X_{i-1}|}\}$ satisfies $\deg(v,X_i)\leq\deg(v_j,X_i)$.This  implies that} the number of edges in $X_i$ satisfies $e(X_i)<\gamma |X_i|^2e(G)/N^2$.
Otherwise, i.e.\ for the case $\deg(v_j,S)\geq 2\gamma |S|e(G)/N^2$, we let 
$X_i=S\setminus N(v_j)$. Then,
\[|X_i|\leq |S|-\deg(v_j,S)=\left(1-2\gamma\frac{e(G)}{N^2}\right)|S|\leq \left(1-2\gamma\frac{e(G)}{N^2}\right)|X_{i-1}|.\]
Finally, observe that it follows from the definition of $v_j$ that we always have $I\setminus\{x_1,\dots,x_{i}\}\subset X_i$, which completes the proof. 
\end{proof}

From this lemma we immediately deduce the following corollaries.

\begin{corollary}\label{cor:supersaturation_ind}
Let $G = (V,E)$ be a fixed $(\lambda$,$\gamma)$-supersaturated graph on $N$ vertices with average degree $D$, where $\lambda, \gamma>0$. Let $t \geq 1$, and $\ell$ be an integer such that $0 < \ell < t$. Finally, set
\[\nu=\nu(\ell)=\max\left\{\left(1- \gamma \frac{D}{N}\right)^{\ell},\lambda\right\}.\]
Then, for every independent set $I\in \mathcal{I}_G(t)$, there exists a subset $L\subset I$ of size $\ell$ and a set $P(L)$, depending only on $L$, of size at most $\nu N$ such that $I\setminus L \subset P(L)\subset V(G)$. Further, we have $L \cap P(L) = \emptyset$. In particular, it follows that $|\mathcal{I}_G(t)|\leq \binom{N}{\ell}\binom{\nu N}{t-\ell}.$
\end{corollary}

\begin{proof}
Given $I\in \mathcal{I}_G(t)$, we apply Lemma \ref{lem:greedyalg}
to obtain a sequence of vertices $x_1,\dots,x_\ell$ 
and sets $V = X_0,X_1, \dots , X_\ell$, as stated. Now set 
$L=\{x_1,\dots,x_\ell\}$ and $P(L)=X_\ell$, and observe that
$I\setminus L\subset P(L)$ and $L\cap P(L)=\emptyset$.

If $|X_i|\leq\left(1-2\gamma\frac{e(G)}{N^2}\right) |X_{i-1}|$ for all 
$i\leq \ell$, then $|P(L)|\leq \left(1-2\gamma \frac{e(G)}{N^2}\right)^\ell N$. In other words, $|P(L)|\leq \left(1-\gamma \frac{D}{N}\right)^\ell N$. On the other hand, if $e(X_i)<\gamma\frac{|X_i|^2}{N^2}e(G)$ for some  $i\leq\ell$,
then $|P(L)|\leq \lambda N$, since by assumption $G$ is $(\lambda,\gamma)$-supersaturated. Altogether, it follows that $|P(L)| \leq \nu N$, which completes the proof.
\end{proof}

\begin{corollary}
\label{cor:stability_ind}
Let $\lambda, \varepsilon,\delta >0$ and $G = (V,E)$ be graph on $N$
vertices which is $(\lambda, \mathcal{B})$-stable with respect to 
$(\varepsilon, \delta)$. 
Let  {$t > \ell \geq \ln\left(\frac{1}{(1-\delta)\lambda}\right)\frac{N^2}{2 \delta e(G)}$}. 
Then, for every independent set $I\in \mathcal{I}_G(t)$, there exists a subset 
$L\subset I$ of size $\ell$ and a set $P(L)\subset V(G)$ depending only on 
$L$ such that $I\setminus L \subset P(L)$ and $L \cap P(L) = \emptyset$. 
Furthermore, either
\begin{itemize}
\item $|P(L)|\leq (1-\delta)\lambda N$, or
\item $|P(L)\setminus B|\leq \varepsilon \lambda N$ for some $B\in \mathcal{B}$.
\end{itemize}
\end{corollary}

\begin{proof}
We apply Lemma \ref{lem:greedyalg} to $G$ using $\gamma = \delta$ in order to obtain a sequence of vertices $x_1, \ldots, x_\ell$ and subsets $X_1, \ldots, X_\ell$ with the desired properties. Let $L = \{x_1, \ldots, x_\ell\}$. If $|X_i|\leq\left(1-2\delta\frac{e(G)}{N^2}\right) |X_{i-1}|$ for all 
$i\leq \ell$, then set $P(L)=X_\ell$. Using our assumption on $\ell$, we get $|P(L)| \leq \left(1-2\delta\frac{e(G)}{N^2}\right)^\ell N\leq(1-\delta)\lambda N$. Otherwise, pick the smallest index $j \leq \ell$ such that $e(G[X_j]) < \delta \frac{|X_j|^2}{N^2} e(G)$, and let $P(L) = X_j$. Again, if $|P(L)| \leq (1 - \delta) \lambda N$ we are done. On the other hand, if this condition does not hold we deduce from the $(\lambda, \mathcal{B})$-stability of $G$ that there exists some $B \in \mathcal{B}$ for which $|P(L) \setminus B| \leq \varepsilon \lambda N$, which completes the proof.
\end{proof}

Now that we have all the necessary machinery, we proceed with the proof of Proposition \ref{prop:transference}.

\begin{proof}[Proof of Proposition~\emph{\ref{prop:transference}}]

Let $\lambda = \lambda(n)$, $\gamma = \gamma(n)$, and 
$\bG=\{G_n\}_{n \in \mathbb{N}}$ be a sequence of graphs. 
For a given $0 < \varepsilon < 1$, let 
$C_1=800/\varepsilon^3$. For the proof of case (\emph{iv}) in Proposition~\ref{prop:transference}, 
suppose that $\bG$ is $(\lambda,\cB)$-stable for some $\cB$. Then
for $\eps'=\eps/2$, there is a constant $\delta' > 0$ and $n_1 \in \mathbb{N}$ such that, for all $n \geq n_1$,
the graph $G_n$ is $(\lambda, \cB_n)$-stable with respect to $(\eps',\delta')$.
We choose $\delta=\min\{\delta'/4,\eps'/{16},1/16\}$ and $C_2=100/\delta^4$. Finally, set $C=\max\{C_1,C_2\}$, and let $n_0 \geq n_1$ be sufficiently large.

We proceed with the proof of the first case of Proposition~\ref{prop:transference}.
Assume that $N^{-1} \ll p \ll D^{-1}$. 
Using the Chernoff bound (Lemma \ref{l:chernoff}),
we have almost surely $|V_p| = (1\pm\eps/2)pN$, which proves the upper bound. 
Further, we have $\ex(e(H_n))=\frac{1}{2}NDp^2$ and by Markov's inequality 
 a.a.s. $e(H_n) \leq \eps pN/2$ holds.
By deleting at most this number of vertices from $H_n$, we obtain an 
independent set of size at least $(1-\eps)pN$, which proves the lower bound.

For the second part, assume that $9D^{-1} \leq   p  \leq  \lambda^{\varepsilon}(\lambda\gamma D)^{-1}$.
Further, let $\ell = (1+\varepsilon)\frac{N}{\gamma D} \ln(pD)>0$, and $t =  \frac{4N}{\eps\gamma D} \ln(pD)$. 
Let $X$ be the random variable counting the number of independent sets of size exactly $t$ in $H_n$, i.e., $X = |\mathcal{I}_{H_n}(t)|$. 
By the choice of our parameters, Corollary \ref{cor:supersaturation_ind} applies, and we obtain:
$$
\mathbb{E}[X] \leq \binom{N}{\ell} \binom{\nu(\ell)N}{t - \ell}p^t,
$$
where $\nu(\ell) = \max\left\{\left(1- \gamma \frac{D}{N}\right)^{\ell},\lambda\right\}$. 
Using $\binom{n}k\leq \left(\frac{en}k\right)^k$ and the choice of $\ell$ and $t$, we get
$$
\binom{N}{\ell} \leq \left(\frac{e \gamma D}{\ln(pD)}\right)^\ell \qquad\text{ and }\qquad
\binom{\nu(\ell)N}{t - \ell} \leq \left (\frac{ e \nu(\ell)\gamma D}{\ln(pD)}  \right)^{t - \ell}.
$$
Combining both inequalities, and noting that our choice of $C$ guarantees that $\ell\leq \varepsilon t/2$, we get:
$$
\mathbb{E}[X]\leq \left(\frac{e\gamma p D \nu^{1-\varepsilon/2}}{\ln(pD)}\right)^{t}.
$$

In case $\nu(\ell)=\lambda$, we have  
$\gamma pD\nu^{1-\varepsilon/2}\leq \lambda^{\eps/2} \leq 1$, since
$p \leq \lambda^{\eps}(\lambda \gamma D)^{-1}$. On the other hand,
if $\nu(\ell) \leq e^{- \ell \gamma \frac{D}{N}} \leq (pD)^{-1-\varepsilon}$,
we have $\gamma pD\nu^{1-\varepsilon/2}\leq \gamma(pD)^{-\varepsilon/2+\varepsilon^2/2}\leq  \gamma \leq 1$ {since $\varepsilon < 1$}.
Hence, $\ex(X) \leq (e/\ln (pD))^t$, and the claim follows from Markov's inequality.

For the third part, assume that $p \geq C(\lambda\gamma D)^{-1}\ln^2 \left (\frac{e}{\lambda} \right )$. Let $t = (1 + \varepsilon) p \lambda N$, and $\ell = \frac{N}{\gamma D} \ln \left ( \frac{e}{\lambda}  \right)$. 
We need to upper bound the following probability:
$$
q = \mathbb{P}[\,\exists I \subset V_p,\,|I| = t,\, I~\text{is an independent set in}~G_n].
$$
It follows from Corollary \ref{cor:stability_ind} that for any $I \in \mathcal{I}_{G_n}(t)$, there exist $L \subset I$ of size $\ell$ and $P(L)$ such that $I \setminus L \subset P(L) \subset V$. Therefore,
$$
q \leq \sum_{L} \mathbb{P}[L \subset V_p~\text{and}~|V_p \cap P(L)| \geq t - \ell].
$$
where the sum is taken over all subsets $L \in \binom{V}{\ell}$ that correspond to some independent as given by Corollary \ref{cor:stability_ind}. 
Using the fact that $L$ and $P(L)$ are disjoint, we obtain
\begin{align}
\label{eq:probq}
q \leq \sum_L \mathbb{P}[L \subset V_p] \cdot \mathbb{P}[|V_p \cap P(L)| \geq t - \ell].
\end{align}

In addition, by our choice of $\ell$, it follows that $\nu(\ell) = \lambda$. Therefore, for any such $L$, we have $|P(L)| \leq \nu(\ell)N \leq \lambda N$. 
Further, the choice of $\ell$ and $p$ implies that $\ell \leq (\varepsilon/2) p \lambda N$.
Hence with $X=|V_p\cap P(L)|$, we have due to the Chernoff bound that
\begin{align*}
 \mathbb{P}(X \geq t - \ell)   
 \leq   \mathbb{P}\left(X \geq p |P(L)|+\frac{\eps p\lambda N}2\right) 
   \leq   \exp\left(-\frac{\eps^2 p\lambda N}{12}\right).
\end{align*} 
From \eqref{eq:probq} and $\binom{N}{\ell} \leq \left (\frac{eN}{\ell} \right )^{\ell}$, it follows  that:
\begin{eqnarray*}
 q  \leq  \left (\frac{eNp}{\ell} \right )^{\ell}\exp\left(-\frac{\eps^2 p\lambda N}{12}\right) 
   =  \exp \left (\ell \cdot \ln \left( \frac{eNp}{\ell} \right) - \frac{\eps^2p\lambda N}{12} \right ).
\end{eqnarray*}
{Recall that we want to prove that $q\leq \exp(-\varepsilon^2 p\lambda N/24)$. With the choice $\ell =\frac{N}{\gamma D}\ln(e/\lambda)$} it is now sufficient to show that {$\frac{1}{\gamma D}\ln(e/\lambda)\ln\left(\frac{ep\gamma D}{\ln(e/\lambda)}\right) \leq \frac{\varepsilon^2p\lambda}{24}$, or equivalently}
\[\frac{24}{\eps^2\gamma\lambda D} \ln (e/\lambda)\leq 
\frac{p}{\ln\left(\frac{ep\gamma D}{\ln (e/\lambda)}\right)}.\]

As the left hand side is independent of $p$, and the right hand side is increasing in $p$, 
it is sufficient to show the inequality for the endpoint
$p=C (\lambda \gamma D)^{-1}\ln^2(e/\lambda)$.
In this case the inequality follows from $24/\eps^2\leq C\ln(e/\lambda) /\ln \left (\frac{eC}{\lambda}\ln(e/\lambda) \right ).$ 
Note that $\ln \left (\frac{eC}{\lambda}\ln(e/\lambda) \right)>\ln (e/\lambda)+\ln C$,  since  $eC/\lambda>\ln (e/\lambda)$. Therefore the bound follows from
$48/\eps^2\leq C\ln (e/\lambda)/\big(\ln (e/\lambda)+\ln C\big)$, {or equivalently  $\frac{48}{\eps^2}\leq \frac{C}{1+\ln (C)/\ln (e/\lambda)}$. Since the right-hand side is decreasing in 
$\lambda$, it is sufficient to verify for $\lambda=1$}, which is immediate from the choice of $C_1$ and $C$.

\medskip

For the last part, let $p \geq C(\lambda D)^{-1}\ln^2 (e/\lambda)$.
Further, let \[\mathcal{T}=\{I\in \mathcal{I}_{G_n}\colon |I|> (1-\delta)\lambda p N \text{ and } |I\setminus B|> \varepsilon \lambda pN \text{ for all } B\in\mathcal{B}_n \}.\] Our task is to upper bound the value of 
\[
q_\mathcal{T}=\pr(\textrm{There is an independent set }I\subset V_p \textrm{ with } I\in \mathcal{T}).
\]

Recall our choice of $\eps'$, $\delta'$, and $n_0$, and that 
$G_n$ is $(\lambda,\mathcal{B}_n)$-stable with respect to $(\varepsilon',\delta')$ for every $n \geq n_0$.
We apply Corollary~\ref{cor:stability_ind} with $\eps'$, $\delta'$,  
$t=(1-4\delta)\lambda p N$, and {$\ell=\frac{N}{\delta D}\ln\frac{e}{\lambda}\leq \delta\lambda pN$. Note that this is a valid choice of $\ell$, since $\frac{N}{\delta D}\ln\frac{e}{\lambda}\geq\ln\left(\frac{1}{(1-\delta)\lambda}\right)\frac{N^2}{2\delta e(G)}$.} This implies that
for every  $I\in \mathcal{T}$  there is some $L=L(I)\subset I$ of size $\ell$ and some $P(L)\subset V(G_n)$, depending  only on $L$ and disjoint from $L$, such that 
$I\setminus L\subset P(L)$. 
Hence, if there is an $I\subset V_p$ with  $I\in \mathcal{T}$, then there is an $L$ of size $\ell$ with
\begin{enumerate}
\item[(A)]  $L\subset V_p$, and 
\item[(B)]  $|P(L)\cap V_p|\geq (1-\delta)p\lambda N-\ell\geq (1-2\delta)p\lambda N \,$ and 
 $\, |(P(L)\setminus B) \cap V_p|>\varepsilon\lambda p N{-\ell\geq \frac{3}{4}\varepsilon\lambda pN}$ for all $B\in\mathcal{B}_n$, {since $\delta\leq \varepsilon'/16 = \varepsilon/32$}.
\end{enumerate}

Let $q_{P(L)}$ {be} the probability that event (B) holds for the random set $V_p$. 
As $L$ and $P(L)$ are disjoint, we  have
\begin{align}
\label{eq:qT}
q_\mathcal{T}\leq \sum_{L}\mathbb{P}[L\subset V_p]\cdot q_{P(L)},
\end{align} 
where the sum ranges over all $L\in\binom{V}\ell$  corresponding to some $I$ as given by Corollary~\ref{cor:stability_ind}.

From Corollary~\ref{cor:stability_ind} and the chosen parameters, either $|P(L)|\leq(1-{\delta'})\lambda N$, or $|P(L)\setminus B|\leq {\varepsilon' \lambda N}$ for 
some $B\in\mathcal{B}_n$. Consider each of the cases separately.
If $|P(L)|\leq {(1 - \delta') \lambda N} \leq (1-4\delta)\lambda N$ then  Chernoff's bound (Lemma \ref{l:chernoff}) 
yields
\[\pr(|P(L)\cap V_p|\geq (1-2\delta)p\lambda N)\leq \exp\{-\delta^2\lambda p N\}.\]
Similarly, if $|P(L)\setminus B|\leq {\varepsilon' \lambda N} = \varepsilon \lambda N/2$ for some $B\in\mathcal{B}_n$, then, together with  $\delta\leq \eps/{32}$, we have
\[\pr\left(|(P(L)\setminus B) \cap V_p|>{\frac{3}{4}}\varepsilon\lambda p N\right) \leq\exp\{{-\eps \lambda p N/{48}}\}\leq \exp\{-\delta^2\lambda pN\}.\]
{Consequently, for every set $L$ as above we have $q_{P(L)} \leq \exp\{-\delta^2\lambda pN\}$.}

Hence \eqref{eq:qT} combined with $\binom{N}\ell\leq \left(\frac{eN}\ell\right)^\ell$ and the choice of $\ell =\frac{N}{\delta D}\ln(e/\lambda)$ yields
\[q_{\mathcal{T}} \;\leq\; \left(\frac{epN}{\ell}\right)^{\ell}\! \exp\{-\delta^2\lambda pN\}\;\leq\; 
 \exp\left\{\ell\ln\left(\frac{{e\delta p D}}{\ln (e/\lambda)}\right)- \delta^2\lambda pN\right\}.\]
To complete the proof it is suffices therefore to show that {$\ell\ln\left(\frac{{e\delta p D}}{\ln (e/\lambda)}\right)< \delta^2\lambda pN/2$, or equivalently}
\[\frac{{2}\ln (e/{\lambda})}{{\lambda\delta^3 D}}< 
\frac{p}{\ln \left(\frac{{e\delta pD}}{\ln(e/\lambda)}\right)}.\]
As the left hand side does not depend on $p$, and the right hand side is monotone increasing in
$p$, it is sufficient to verify this inequality for the endpoint $p=C(\lambda D)^{-1}\ln^2(e/\lambda)$. In 
this case and noting that  ${e \delta C}/\lambda>\ln (e/\lambda)$ due to our choice of $C_2$ and $C$, the claim follows from
\[\frac{{2}}{{\delta^3}}< \frac{C\ln(e/\lambda)}{\ln \left(\frac{{e \delta C}}{\lambda}\ln (e/\lambda)\right)}< \frac{C\ln(e/\lambda)}{2\ln\left(\frac{e\delta C}{\lambda}\right)}=\frac{C}{2+2\ln(\delta C)/\ln(e/\lambda)}.\]  
As the right-hand side is decreasing in $\lambda$ it is sufficient to verify for  $\lambda=1$ which, however, is immediate  from the choice of  $C$ and $C_2$. 
This completes the proof.
\end{proof}

\section{Concluding remarks}\label{s:concluding}
While this work was under review there has been a vivid interest in questions related to random versions of the 
Erd\H{o}s-Ko-Rado theorem~(cf. \cite{BDDLS,HammKahn1,HammKahn2, BNR,BBN,DevlinKahn, DasTran}).
In particular, besides the results of Balogh, Bohman, and Mubayi~\cite{BBM},  
the question concerning the structure of the largest intersecting family in the random setting has been addressed
in~\cite{HammKahn1,HammKahn2,BDDLS} for various ranges of $k$ and $p$.
Moreover, an extension of the robust stability result for intersecting families, Lemma~\ref{lem:robuststability}, has been considered in~\cite{DasTran}, implying
that Theorem~\ref{thm:stability} can be extended to a larger range of $k$. We refer to these papers for further information.

\section*{Acknowledgement}
The authors are grateful to Yoshiharu Kohayakawa for mentioning the problem and for several helpful comments.
Many thanks go to Oded Regev and Ehud Friedgut for sharing a preliminary version of their paper. Finally, we thank the anonymous referee for carefully reading the paper and for many suggestions that improved its readability.
\bibliographystyle{amsplain}	
\bibliography{ekr_lit}

\end{document}